\documentclass[a4paper, dvipdfmx]{amsart}

\usepackage{amsmath, amsthm, bm, bbm, amssymb, extarrows, mathrsfs}
\usepackage{fancyhdr}
\usepackage{ascmac}
\usepackage{color}
\usepackage{blkarray}
\usepackage{url}
\usepackage{hyperref}
\usepackage{cleveref}
\usepackage{autonum}

\usepackage{tikz}
\usetikzlibrary{intersections, calc, arrows.meta}

\crefname{equation}{Eq.}{Eqs.}

\theoremstyle{plain}
\newtheorem{thm}{Theorem}[section]
\crefname{thm}{Theorem}{Theorems}
\newtheorem{prop}{Proposition}[section]
\newtheorem{lem}{Lemma}[section]

\newtheorem{prob}{Problem}[section]

\theoremstyle{definition}
\newtheorem{df}{Definition}[section]
\newtheorem{rem}{Remark}[section]

\newcommand{\C}{\mathcal{C}}
\newcommand{\wt}[1]{{\rm wt}(#1)}

\DeclareMathOperator{\Spec}{Spec}

\title{Generalized graph codes and their minimum distances}
\author[N.~Fujii]{Naoki Fujii}
\thanks{Graduate School of Fundamental Science and Engineering, Waseda University}
\email{naoki.fujii@fuji.waseda.jp}
\date{2024/07/01}

\begin{document}

\begin{abstract}
    {\it Graph code} is a linear code obtained from linear codes \(C\) and a certain bipartite graph \(G\). In this paper, I propose an expansion of the definition of graph code to general \(\ell\)-partite, and give its lower bound of minimum distance. I also give an example of generalized graph code and calculate its parameters \([n,k,d]\). 
\end{abstract}

\maketitle

\section{Introduction}
For ordinary graph code is defined as follow\cite{Tanner, Hoholdt--Justesen}: Let \(G = (V, E)\) be a \(n\)-regular bipartite graph on \(V = V_1 \sqcup V_2\) with \(|V_1| = |V_2| = m\) and ordering edges \(e \in E\) in some way, \(C_1\) be a \([n, k_1, d_1]\)-linear code over \(\mathbb{F}_q\), and \(C_2\) be a \([n, k_2, d_2]\)-linear code over \(\mathbb{F}_q\). Considering that giving a symbol of \(\mathbb{F}_q\) for all edges arbitrarily. Then, the sequence of \(mn\) symbols indexed by edges with given order can be regarded as a kind of code word. Let \(c(v)\) be a induced codeword from \(c\) (the sequence above) by reading the symbols on the edges connecting to \(v\). If two sequences \(c_1, c_2\) satisfy following two condition:
\begin{enumerate}
    \item For all \(v_1 \in V_1\), \(c_1(v_1) \in C_1\) and \(c_2(v_1) \in C_1\),
    \item For all \(v_2 \in V_2\), \(c_1(v_2) \in C_2\) and \(c_2(v_2) \in C_2\),
\end{enumerate}
then \((c_1 + c_2)(v)\) is a codeword of \(C_1\) for all \(v \in V_1\), and is a codeword of \(C_2\) for all \(v \in V_2\).


\begin{df}[\cite{Hoholdt--Justesen}]
    We say the following linear code \(\C\) is a {\it graph code}.
    \begin{align}
        \C = \{c \in \mathbb{F}_q^{mn} \mid \text{\(c(v_1) \in C_1\) for every \(v_1 \in V_1\) and \(c(v_2) \in C_2\) for every \(v_2 \in V_2\)}\}
    \end{align}
\end{df}

The minimum distance of graph code \(\C\) is bounded with the second largest adjacent eigenvalue \(\lambda_2\) of foundataion graph \(G\) as
\begin{prop}[\cite{Hoholdt--Justesen}]\label{Hoholdt_Justesen_minimum}
    Let \(D\) be the minimum distance of \(\C\) with above condition. Then, 
    \begin{align}
        D \geq dm\frac{d-\lambda_2}{n-\lambda_2}. \label{eq1:bound2}
    \end{align}
\end{prop}

\Cref{Hoholdt_Justesen_minimum} is obtained by Sipser and Spielman\cite{Sipser--Spielman} with another form as a bound for expander codes, and the specific version for graph code is introduced in \cite{Hoholdt--Justesen} as the above inequality (\ref{eq1:bound2}).

In this paper, firstly, graph code will be generalized into \(\ell\)-partite graph, and the same type of  proposition on the minimum distance will be showed: 
\begin{thm}\label{main-theorem}
    Let \(\C\) be a generalized graph code on \(\ell\)-partite graph with \([(\ell - 1)n,k,d]\)-linear code \(C\). Then, \(D := d(\C)\) satisfies
    \begin{align}
        D \geq \frac{dm(d-\lambda_2)}{(\ell - 1)n - \lambda_2}.
    \end{align} 
\end{thm}

This paper is organized as follows: In Section \ref{Preliminaries}, I define the generalized graph code on \(\ell\)-partite graph and codes. I also review some needed properties of them. In Section \ref{proof-main}, I show the main Theorem \ref{main-theorem}. In section \ref{eg-code}, I introduce some example of generalized graph code.

\section{Preliminaries}\label{Preliminaries}

In this section, I introduce needed facts including the proposal of generalized graph code.

\subsection{Linear code}

Let \(\mathbb{F}_q\) be a finite field with order \(q\). \(\mathbb{F}_q^n\) is a linear space over \(\mathbb{F}_q\) with the addtion and the scalar multiplication induced from the operations of \(\mathbb{F}_q\). Consider a subspace \(C < \mathbb{F}_q^n\). This \(C\) is called a linear code over \(\mathbb{F}_q\). For \(x,y \in C\), Hamming distance \(d(x,y)\) is the number of cordinates in which the entries of \(x, y\) are different, i.e. \(d(x,y) := \#\{i \in \{1,2,\ldots, n\} \mid x_i \neq y_i\}\). Hamming distance between \(x\) and \(0\) (\( = d(x,0)\)) is called {\it weight} of \(x\), and denoted as \(\wt{x}\). The minimum distance of \(C\) is defined as \(\min_{x \neq y} d(x,y)\), where \(x,y \in C\) and denoted \(d(C)\). For the linearity of \(C\), the minimum distance is equal to the minimum weight of the elements in \(C\).

If a code \(C\) is a subspace of \(\mathbb{F}_q^n\), has dimension \(k = \dim C\), and the minimum distance \(d(C)\) is \(d\), then we refer \(C\) as a \([n,k,d]\)-linear code, or in this paper, I sometimes use the notation \(C = [n,k,d]\).

\subsection{\(\ell\)-partite and generalized graph code}

I give the definition of \(\ell\)-partite graph. Considering the definition of bipartite graph, it is natural to increase the number of components of vertex set \(V\) to generalize.
\begin{df}
    Let \(G = (V, E)\) be a graph. If the vertex set \(V\) has a from of disjoint union \(V = V_1 \sqcup \cdots \sqcup V_\ell\) such that
    \begin{align}
        \text{for all } u,v \in V_i,\ \{u,v\} \not\in E        
    \end{align}
    holds for each \(i = \{1,2,\ldots, \ell\}\), then we say that \(G\) is {\it \(\ell\)-partite}.
\end{df}

I propose {\it generalized graph code} as a expansion of graph code to \(\ell\)-partite graph: Let \(G = (V, E)\) be a \(\ell\)-partite graph with the vertex decomposition \(V = V_1 \sqcup \cdots \sqcup V_\ell\) and satisfy
\begin{enumerate}
    \item[(i)] for each component \(V_i\) of the decomposition, \(|V_1| = \cdots = |V_\ell| = m\), and
    \item[(ii)] for all \(i \in \{1,2, \ldots, \ell\}\) and \(v \in V_i\), \(n = |N(v) \cap V_j|\) is not depend on the choice of \(j \neq i\) and \(v\), where \(N(v) = \{u \in V \mid \{u,v\} \in E\}\), 
\end{enumerate}
and insert the total order (consider labeling with natural numbers) into the edge set \(E\) of \(G\) above.

\begin{rem}
    The condition (ii) above claims that each vertex \(v\) has precisely \(n\) edges to each component \(V_i\). Clearly, \(G\) is \((\ell - 1)n\)-regular.
\end{rem}

Prepare \(\ell\) codes \(C_1 = [(\ell - 1)n, k_1, d_1]\), \(C_2 = [ (\ell - 1)n, k_2, d_2]\), \(\ldots\), \(C_\ell = [(\ell - 1)n, k_\ell, d_\ell]\). Now, let us consider allocating the elements of \(\mathbb{F}_q\) (hereafter referred to as {\it \(\mathbb{F}_q\)-symbols} or simply {\it symbols}) to the edges of \(G\). 
Consider \(\ell\) conditions below:
\begin{enumerate}
    \item[(\(\ast1\))] For all \(u_1 \in V_1\), \(c(u_1)\) is a codeword of \(C_1\).
    \item[(\(\ast2\))] For all \(u_2 \in V_2\), \(c(u_2)\) is a codeword of \(C_2\).
    \item[\(\vdots\)]
    \item[(\(\ast \ell\))] For all \(u_\ell \in V_\ell\), \(c(u_\ell)\) is a codeword of \(C_\ell\). 
\end{enumerate}

The entire of allocations of \(\mathbb{F}_q\)-symbols (let us denote \(\C\)) satisfying conditions above forms a \([\ell(\ell-1) mn/2]\)-length linear code over \(\mathbb{F}_q\) (codewords are an arrangement of symbols on all edges arranged in the above order), then we call it {\it generalized graph code}. 
\begin{align}
    \C = \{c \in \mathbb{F}_q^{\ell(\ell-1) mn/2} \mid \text{\(c\) satisfies (\(\ast 1\))\(\sim\)(\(\ast \ell\))}\}
\end{align}

\section{Proof of Theorem \ref{main-theorem}} \label{proof-main}

Recall Theorem \ref{main-theorem} with the accurate description:
{\let\temp\thethm
\renewcommand{\thethm}{\ref{main-theorem}}
\begin{thm}[Recall]
    Let \(\C = [\ell(\ell - 1) mn/2 , K, D]\) be a generalized graph code on \(\ell\)-partite graph \(G = (V, E)\) (valency and vertices conditions are above) with prepared codes \(C_1 = C_2 = \cdots = C_\ell = [(\ell - 1)n,k,d]\). Then, 
    \begin{align}
        D \geq \frac{dm(d-\lambda_2)}{(\ell - 1)n - \lambda_2}
    \end{align}
    holds, where \(\lambda_2\) is the second largest adjacecy eigenvalues of \(G\).
\end{thm}}

\begin{proof}[Proof of Theorem \ref{main-theorem}]

    Let \(c \in \C\) be a codeword that reach the minimum distance. For \(i \in \{1, 2, \ldots, \ell\}\), take
    \begin{align}
        S_i := {\rm Supp}(c) \cap V_i.
    \end{align}
    For the symmetricity of \(G\), we can assume \(a = |S_1| \geq |S_2| \geq \cdots \geq |S_\ell|\). If \(|S_i| < a\), choose vertices from \(V_i \backslash S_i\) and put them into \(S_i\) to make \(|S_i| = a\). We use the notation \(\overline{S_i}\) as \(V_i \backslash S_i\). Let \(x_i\) be a vector indexed by \(V_i\), with \(m-a\) at the coordinate corresponding to \(S_i\) and \(-a\) at other coordinates, and \(A\) be the adjacecy matrix of \(G\). By denoting the adjacency relationship from \(V_i\) to \(V_j\) as \(A_{i,j}\), \(A\) can be written as follow:
    \begin{align}
        A = \left[
            \begin{array}{ccccc}
                O & A_{1,2} & A_{1,3} & \cdots & A_{1,\ell}\\
                A_{2,1} & O & A_{2,3} & \cdots & A_{2, \ell}\\
                A_{3,1} & A_{3,2} & O & \cdots & A_{3, \ell}\\
                \vdots & \vdots & \vdots & \ddots & \vdots\\
                A_{\ell, 1} & A_{\ell, 2} & A_{\ell, 3} & \cdots & O
            \end{array}
        \right].
    \end{align}
    
    To prove the main theorem, consider \(x^\top A x\), where \(x\) be the column vector that is obtained by arranging \(x_i\). For the representation of \(A\) above, we have
    \begin{align}
        x^\top A x = \sum_{i \in \{1, \ldots, \ell\}} \sum_{j:j \neq i} x_i^\top A_{i,j} x_j.
    \end{align}

    Here, note that \(x_i^\top A_{i,j} x_j\) is the sum of the products of two values of \(x\) at corresponding coordinates to the endpoints of edges between \(V_i\) and \(V_j\).

    Now, calculate \(\sum_{j} x_i^\top A_{i,j} x_j\) with fixed \(i\). To simplify the discussion, assume that \(i = 1\). In the following discussion, I use the notation \(E(U, T)\) as the edge set between \(U\) and \(T\).

    \begin{enumerate}
        \item[(i)] Consider the edges between \(S_1\) and \(S_2 \cup \cdots \cup S_\ell\). For arbitrarily chosen \(u \in S_1\), \(u\) has at least \(d\) edges to \(S_2 \cup \cdots \cup S_\ell\) because the number of these edges should be larger than the miniimum distance of \(C\). Therefore, \(| E(\{u\}, S_2 \cup \cdots \cup S_\ell) | \geq d\), and we have
        \begin{align}
            | E(S_1, S_2 \cup \cdots \cup S_\ell) | = \sum_{u \in S_1} | E(\{u\}, S_2 \cup \cdots \cup S_\ell) | \geq \sum_{u \in S_1} d = da.
        \end{align}

        \vspace{10 pt}

        \item[(ii)] Consider the edges between \(S_1\) and \(\overline{S_2} \cup \cdots \cup \overline{S_\ell}\). For arbitrarily chosen \(u \in S_1\), \(u\) has at most \((\ell-1)n-d\) edges to \(\overline{S_2} \cup \cdots \cup \overline{S_\ell}\) (\(u\) has \((\ell-1)n\) edges to \(S_2 \cup \overline{S_2} \cup \cdots \cup S_\ell \cup \overline{S_\ell}\)). Therefore, we have
        \begin{align}
            | E(S_1, \overline{S_2} \cup \cdots \cup \overline{S_\ell}) | = \sum_{u \in S_1} | E(\{u\}, \overline{S_2} \cup \cdots \cup \overline{S_\ell}) | &\leq \sum_{u \in S_1} ((\ell-1)n-d)\\
            &= (\ell - 1)an - da. \nonumber
        \end{align}
        
        \vspace{10 pt}

        \item[(iii)] Consider the edges between \(\overline{S_1}\) and \(S_2 \cup \cdots \cup S_\ell\). Firstly, count the number of edges between \(V_1\) and \(S_2 \cup \cdots \cup S_\ell\). Take a vertex \(u \in S_2 \cup \cdots \cup S_\ell\). For regularity of \(G\), \(u\) exactly has \(n\) edges to \(V_1\). Then, 
        \begin{align}
            | E(V_1, S_2 \cup \cdots \cup S_\ell) | = \sum_{u \in S_2 \cup \cdots \cup S_\ell} \deg(u) = (\ell - 1)an.
        \end{align}
        Here, \(|E(V_1, S_2 \cup \cdots \cup S_\ell)| = |E(S_1, S_2 \cup \cdots \cup S_\ell)| + |E(\overline{S_1}, S_2 \cup \cdots \cup S_\ell)|\), then from (i), we have
        \begin{align}
            |E(\overline{S_1}, S_2 \cup \cdots \cup S_\ell)| &= |E(V_1, S_2 \cup \cdots \cup S_\ell)| - |E(S_1, S_2 \cup \cdots \cup S_\ell)|\\
            & \leq (\ell - 1)an - da. \nonumber
        \end{align}
        

        \vspace{10 pt}

        \item[(iv)] Consider the edges between \(\overline{S_1}\) and \(\overline{S_2} \cup \cdots \cup \overline{S_\ell}\). For arbitrarily chosen \(u \in \overline{S_1}\), \(u\) has \((\ell-1)n\) edges to \(V_2 \cup \cdots \cup V_\ell\). Therefore, 
        \begin{align}
            | E(\overline{S_1}, \overline{S_2} \cup \cdots \cup \overline{S_\ell}) | &= | E(\overline{S_1}, V_2 \cup \cdots \cup V_\ell) | - | E(\overline{S_1}, S_2 \cup \cdots \cup S_\ell) | \\
            &= \sum_{u \in \overline{S_1}} (\ell-1)n - \sum_{u \in \overline{S_1}} | E(\{u\}, S_2 \cup \cdots \cup S_\ell) | \nonumber\\
            &\geq (\ell - 1)(m-a)n - ((\ell - 1)an - da) \nonumber\\
            &= (\ell - 1)mn - 2(\ell - 1)an + da. \nonumber
        \end{align} 

    \end{enumerate}

    \vspace{10 pt}
    For these bounds of the number of edges, we have
    \begin{align}
        \sum_{j = 2}^\ell &x_1^\top A_{1, j} x_j\\
        &\geq da \cdot (m-a)^2 + \left[(\ell - 1)an - da + (\ell - 1)an - da\right] \cdot (-a)(m-a)\nonumber\\
        & \hspace{70 pt} + ((\ell - 1)mn - 2(\ell - 1)an + da) \cdot (-a)^2 \nonumber\\
        &= m^2da-(\ell-1)mna^2 \nonumber
    \end{align}

    From the symmetricity of \(G\), we have the same result with \(i = 2, 3, \ldots, \ell\). Therefore, 
        \begin{align}
            x^\top A x \geq \ell(m^2da-(\ell-1)mna^2) \label{eq1}
        \end{align}
    holds.

    To complete the proof, the following lemma is useful.
    \begin{lem}
        Let \(P\) be a \(i \times i\) real symmetric matrix with an eigenvector \(x\) corresponding to the largest eigenvalue of \(P\), and the second largest eigenvalue \(\theta_2\). For any \(y \in \mathbb{R}^i\) orthogonal to \(x\), 
        \begin{align}
            y^\top Py \leq \theta_2|y|^2
        \end{align}
        holds.
    \end{lem}

    \begin{proof}
        Take \(\Spec(P)\), the spectrum of \(P\), and corresponding eigenvectors \(x_j\) of \(\theta_j \in \Spec(P)\). Note that \(x = x_1\). Here, we take \(|x_j| = 1\) for any \(j = 1, \ldots, i\). Since \(y \perp x = x_1\), the vector \(y\) can be expanded as \(y = a_2 x_2 + a_3 x_3 + \cdots + a_ix_i\) for some \(a_2, \ldots, a_i \in \mathbb{R}\). Thus, we have
        \begin{align}
            y^\top P y = \sum_{j = 2}^i \theta_j a_j^2 \leq \theta_2 \sum_{j = 2}^i a_j^2 = \theta_2 |y|^2 
        \end{align}
    \end{proof}

    Since \(G\) is regular, the largest adjacent eigenvector is all-\(1\) vector \(\mathbbm{1}\). The vector \(x\) in this proof is orthogonal to the corresponding eigenvector to the largest eigenvalue of \(A\). Thus, we have
    \begin{align}
        x^\top A x \leq \lambda_2 |x|^2 = \lambda_2 \ell am(m-a), \label{eq2}
    \end{align}
    where \(\lambda_2\) is the second largest eigenvalue of \(A\). 

    From (\ref{eq1}) and (\ref{eq2}), we finally have
    \begin{align}
        a \geq \frac{m(d-\lambda_2)}{(\ell - 1)n - \lambda_2}.
    \end{align}

    Now, from the construction of \(\C\), \(D \geq da\) holds. Therefore, we have
    \begin{align}
        D \geq \frac{dm(d-\lambda_2)}{(\ell - 1)n - \lambda_2},
    \end{align}
    which is desired.
\end{proof}

\section{Examples of Generalized Graph Code} \label{eg-code}

\subsection{A \(3\)-partite graph code on \(K_{7,7,7}\)} 
Let \(G\) be a \(3\)-partite graph with the vertex set \(V = V_1 \cup V_2 \cup V_3\) where 
\begin{align}
    V_1 &= \{v_1, v_2, v_3, v_4, v_5, v_6, v_7\}\\
    V_2 &= \{v_8, v_9, v_{10}, v_{11}, v_{12}, v_{13}, v_{14}\}\\
    V_3 &= \{v_{15}, v_{16}, v_{17}, v_{18}, v_{19}, v_{20}, v_{21}\}
\end{align}
and all two distinct vertices are adjacent. Then, \(G\) is complete \(3\)-partite \(K_{7,7,7}\). Define an order \(\succ_V\) on \(V\) such that \(v_i \succ_V v_j\) if \(i < j\). Then, consider the order \(\succ\) on the edges set \(E\) that is induced by \(\succ_V\) following: for any two edges \(e_1 = \{u,v\}\) and \(e_2 \{w,x\}\), \(e_1 \succ e_2\) if
\begin{enumerate}
    \item \(\max\{u,v\} \succ_V \max\{w,x\}\),
    \item[] or
    \item \(\max\{u,v\} = \max\{w,x\}\) and \(\min\{u,v\} \succ_V \min\{w,x\}\).
\end{enumerate}

Take binary \([7,4,3]\)-Hamming code \(H\) (fot the fundamental properties, see for example \cite{Fundamentals}), and set
\begin{align}
    C = \{[1,0] \otimes c_1 + [0,1] \otimes c_2 \mid c_1, c_2 \in H\},
\end{align}
where ``\(\otimes\)'' is Kronecker product. This \(C\) can be assumed as a \(14\)-length linear code obtained by placing codewords of \(H\) side by side.

In this section, we consider the graph code on \(G = K_{7,7,7}\) with based code \(C = [H, H]\).

\subsubsection{Dimension of \(C\)}
Let \(e_1, \ldots, e_4\) be basis of \([7,4,3]\)-Hamming code \(H\), \(C\) can be spanned by \([1,0] \otimes e_1, \ldots, [1,0] \otimes e_4\) and \([0,1] \otimes e_1, \ldots, [0,1] \otimes e_4\). Then, the dimension of \(C\) is \(8\). 
\subsubsection{Minimum distance of \(C\)} Let \(c_0 \in H\) be a codeword of \(H\) that gives the minimum distance of \(H\). Then, \(\wt{c_0} = 3\). \([1,0] \otimes c_0\) is a codeword of \(C\) with weight \(3\), thus, the minimum distance of \(C\) is \(3\).

Next, consider \(d(\C)\). The simple and straightforward way is determining the parameter \([n,k,d]\) of \(\C\). From the construction, \(n = 147\) is easy to see. So, let calculate rest \(k,d\). 

\subsubsection{Dimension of generalized graph code \(\C\)}
Let us denote the element on \(v_i\) as \(s_i\). Consider three \(7 \times 7\) matirces \(M_{12}\), \(M_{13}\), \(M_{23}\) below:
\begin{align}
    M_{12} &= (s_is_j)_{\substack{1 \leq i \leq 7 \\ 8 \leq j \leq 14}}\\
    M_{13} &= (s_is_j)_{\substack{1 \leq i \leq 7 \\ 15 \leq j \leq 21}}\\
    M_{23} &= (s_is_j)_{\substack{8 \leq i \leq 14 \\ 15 \leq j \leq 21}}
\end{align}
These matirces are kinds of a ``list'' of symbols on \(E\). Therefore, if we read the entries of \(M_{12}\), \(M_{13}\), and \(M_{23}\) in order \(\succ\), this is a codeword of \(\C\). Then, \(c \in \C\) can be considered as 
\begin{align}
    c = [1,0,0] \otimes M_{12} + [0,1,0] \otimes M_{13} + [0,0,1] \otimes M_{23}.
\end{align}
Suppose that the allocation of elements on each \(V_1\), \(V_2\), and \(V_3\) form the base vectors of \(H\). Then, denote them as \(e_{V_1}, e_{V_2}, e_{V_3}\), each \(M_{ij}\) gained from \(e_{V_1}, e_{V_2}, e_{V_3}\) form a basis of \(\C\) such that
\begin{align}
    \left\{
        [1,0,0] \otimes M_{12}, [0,1,0] \otimes M_{13}, [0,0,1] \otimes M_{23} 
    \right\}
\end{align}
Since each matrix \(M\) has \(4 \times 4 = 16\) (the choice of \(e_1, \ldots, e_4\)) types, then the dimension of \(\C\) is \(48\).

\subsubsection{Minimum distance of generalized graph code \(\C\)}
Suppose that the each \(7\)-length codewords on \(V_1\) and \(V_2\) give the minimum distance of \(H\), and all symbols on \(V_3\) is \(0\). Then, the graph codeword gained by this allocation gives the minimum distance of \(\C\). Here, \(M_{13}\) and \(M_{23}\) are all-\(0\) matrices, and \(M_{12}\) has precisely \(9\) entries that is non-zero. Thus, the minimum distance of \(\C\) is \(9\). This satisfies Theorem \ref{main-theorem}:
\begin{align}
    9 \geq \frac{3 \cdot 7 \cdot (3 - 0)}{(3-1) \cdot 7 - 0} = \frac{9}{2}
\end{align}

From the above discussion, we gain a \([147, 48, 9]\)-linear code on \(K_{7,7,7}\) with \(C = [1,0] \otimes H + [0,1] \otimes H\).

\subsection{A \(3\)-partite graph code on \(K_{3,3,3}\) with a ``even weight code''}

Consider the complete tri-partite \(K_{3,3,3}\) with vertices partition \(V = V_1 \sqcup V_2 \sqcup V_3\). On this graph (hereafter, denoted as \(G\)), construct a graph code with a \([6, 5, 2]\)-linear code
\begin{align}
    C = \{c \in \mathbb{F}_2^6 \mid \text{\(\wt{c}\) is even.}\}
\end{align}
Let \(\C\) be an obtained graph code. From the definition, since \(\wt{c(v)}\) is not depend on the coordinates order, if \(c \in \C\) is a graph codeword with a edge-order, \(c\) alse is a graph codeword with another order. Therefore, we does not have to consider the edge order. 

Consider \(9 \times 9\) matrix \(M\) (see the matrix \(M\) below) indexed by \(V(G)\) corresponding symbol allocation, that is the matrix which \((v_i, v_j)\)-entry is the symbol on \(\{v_i, v_j\} \in E(G)\) (and for ``diagonal'' parts corresponding to partition components, we set all entries are \(0\)).

    \begin{align}
        M = \begin{blockarray}{cccccccccccc}
             & & v_1 & v_2 & v_3 & v_4 & v_5 & v_6 & v_7 & v_8 & v_9 \\
            \begin{block}{c[cccc|ccc|cccc]}
                v_1 & & 0 & 0 & 0 & 1 & 0 & 0 & 0 & 0 & 1 & \\
                v_2 & & 0 & 0 & 0 & 0 & 1 & 1 & 1 & 0 & 1 & \\
                v_3 & & 0 & 0 & 0 & 0 & 1 & 0 & 1 & 1 & 1 & \\ \cline{3-11}
                v_4 & & 1 & 0 & 0 & 0 & 0 & 0 & 1 & 1 & 1 & \\
                v_5 & & 0 & 1 & 1 & 0 & 0 & 0 & 0 & 1 & 1 & \\
                v_6 & & 0 & 1 & 0 & 0 & 0 & 0 & 1 & 1 & 1 & \\ \cline{3-11}
                v_7 & & 0 & 1 & 1 & 1 & 0 & 1 & 0 & 0 & 0 & \\
                v_8 & & 0 & 0 & 1 & 1 & 1 & 1 & 0 & 0 & 0 & \\
                v_9 & & 1 & 1 & 1 & 1 & 1 & 1 & 0 & 0 & 0 & \\
            \end{block}
        \end{blockarray}
    \end{align}
\begin{center}
    (A matrix \(M\) corresponding to a proper symbols allocation. )
\end{center}

\vspace{20 pt}

Now, let us consider the minimum distance of \(\C\). The sum of entries of \(M\) is
\begin{align}
    \sum_{1 \leq i,j \leq 9} M(v_i, v_j) &= \sum_{v_i \sim v_j} M(v_i, v_j)\\
    &= 2\sum_{\{v_i, v_j\} \in E} c_{\{v_i, v_j\}}\\
    &= 2\wt{c}
\end{align}
Since each \(v_i\)-th row sum is coincide \(\wt{c(v_i)}\), we have
\begin{align}
    \sum_{1 \leq i,j \leq 9} M(v_i, v_j) = \sum_{1 \leq i \leq 9} \wt{c(v_i)}.
\end{align}
So, we should think about the minimum value of the weight sum of \(c(v_i)\).
\begin{enumerate}
    \item If only two edges \(\{v_1, v_4\}\) and \(\{v_1, v_5\}\) have symbol \(1\), and others have symbol \(0\), we have \(\wt{c(v_4)} = 1\). That does not satisfy the definition of \(\C\). In the case of not symbol \(1\) on \(\{v_1, v_5\}\) but on \(\{v_1, v_7\}\), the same violation occurs.
    \item If only three edges \(\{v_1, v_4\}\), \(\{v_4, v_7\}\), and \(\{v_1, v_7\}\) have symbol \(1\), then
    \begin{enumerate}
        \item \(\wt{c(v_1)} = \wt{c(v_4)} = \wt{c(v_7)} = 2\).
        \item Other \(c(v)\) have weight \(0\).
    \end{enumerate}
    So, this situation is suitable. Then, \(\wt{c} = (2 + 2 + 2)/2 = 3\).
\end{enumerate}
For similarity of \(K_{3,3,3}\), we have the minimum distance of \(\C\) is \(d(\C) = 3\). Since \(\lambda_2(K_{3,3,3}) = 0\), that fits
\begin{align}
    3 \geq \frac{2 \cdot 3 \cdot (2 - 0)}{(3 - 1) \cdot 3 - 0} = 2.
\end{align}

\section{Final remark}

I solved the bound for minimum distance of generalized graph code on \(\ell\)-partite graphs theoretically. However, I have not obtained an example that reach the lower bound. So, I am interested in

\begin{prob}
    Is there an example of a generalized graph code that reach the lower bound for \(\ell \geq 3\)? If not, is the bound reached asymptotically? 
\end{prob}

\begin{prob}
    Can the bound in \cref{main-theorem} be more strictly?  
\end{prob}


\begin{thebibliography}{1}

    \bibitem{Hoholdt--Justesen}
    Tom H{\o}holdt and J{\o}rn Justesen.
    \newblock The minimum distance of graph codes.
    \newblock In Yeow~Meng Chee, Zhenbo Guo, San Ling, Fengjing Shao, Yuansheng
      Tang, Huaxiong Wang, and Chaoping Xing, editors, {\em Coding and Cryptology},
      pages 201--212, Berlin, Heidelberg, 2011. Springer Berlin Heidelberg.
    
    \bibitem{Fundamentals}
    W.C. Huffman and V.~Pless.
    \newblock {\em Fundamentals of Error-Correcting Codes}.
    \newblock Cambridge University Press, 2003.
    
    \bibitem{Sipser--Spielman}
    M.~Sipser and D.A. Spielman.
    \newblock Expander codes.
    \newblock {\em IEEE Transactions on Information Theory}, 42(6):1710--1722,
      1996.
    
    \bibitem{Tanner}
    R.~Michael Tanner.
    \newblock A recursive approach to low complexity codes.
    \newblock {\em IEEE Trans. Inform. Theory}, 27(5):533--547, 1981.
    
    \end{thebibliography}
\end{document}